\documentclass[12pt]{amsart}
\usepackage{mathrsfs}
\usepackage{amssymb}
\usepackage{verbatim}
\usepackage{hyperref}
\usepackage{color}
\usepackage{amsfonts}
\usepackage{mathrsfs}
\usepackage{amsmath}
\usepackage{amssymb}
\usepackage{bbm}

\begin{document}
\newtheorem{theorem}{Theorem}
\newtheorem{proposition}[theorem]{Proposition}
\newtheorem{conjecture}[theorem]{Conjecture}
\def\theconjecture{\unskip}
\newtheorem{corollary}[theorem]{Corollary}
\newtheorem{lemma}[theorem]{Lemma}
\newtheorem{sublemma}[theorem]{Sublemma}
\newtheorem{observation}[theorem]{Observation}
\theoremstyle{definition}
\newtheorem{definition}{Definition}
\newtheorem{notation}[definition]{Notation}
\newtheorem{remark}[definition]{Remark}
\newtheorem{question}[definition]{Question}
\newtheorem{questions}[definition]{Questions}
\newtheorem{example}[definition]{Example}
\newtheorem{problem}[definition]{Problem}
\newtheorem{exercise}[definition]{Exercise}

\numberwithin{theorem}{section} \numberwithin{definition}{section}
\numberwithin{equation}{section}

\def\earrow{{\mathbf e}}
\def\rarrow{{\mathbf r}}
\def\uarrow{{\mathbf u}}
\def\varrow{{\mathbf V}}
\def\tpar{T_{\rm par}}
\def\apar{A_{\rm par}}

\def\reals{{\mathbb R}}
\def\torus{{\mathbb T}}
\def\heis{{\mathbb H}}
\def\integers{{\mathbb Z}}
\def\naturals{{\mathbb N}}
\def\complex{{\mathbb C}\/}
\def\distance{\operatorname{distance}\,}
\def\support{\operatorname{support}\,}
\def\dist{\operatorname{dist}\,}
\def\Span{\operatorname{span}\,}
\def\degree{\operatorname{degree}\,}
\def\kernel{\operatorname{kernel}\,}
\def\dim{\operatorname{dim}\,}
\def\codim{\operatorname{codim}}
\def\trace{\operatorname{trace\,}}
\def\Span{\operatorname{span}\,}
\def\dimension{\operatorname{dimension}\,}
\def\codimension{\operatorname{codimension}\,}
\def\nullspace{\scriptk}
\def\kernel{\operatorname{Ker}}
\def\ZZ{ {\mathbb Z} }
\def\p{\partial}
\def\rp{{ ^{-1} }}
\def\Re{\operatorname{Re\,} }
\def\Im{\operatorname{Im\,} }
\def\ov{\overline}
\def\eps{\varepsilon}
\def\lt{L^2}
\def\diver{\operatorname{div}}
\def\curl{\operatorname{curl}}
\def\etta{\eta}
\newcommand{\norm}[1]{ \|  #1 \|}
\def\expect{\mathbb E}
\def\bull{$\bullet$\ }

\def\xone{x_1}
\def\xtwo{x_2}
\def\xq{x_2+x_1^2}
\newcommand{\abr}[1]{ \langle  #1 \rangle}

\newcommand{\Norm}[1]{ \left\|  #1 \right\| }
\newcommand{\set}[1]{ \left\{ #1 \right\} }
\def\one{\mathbf 1}
\def\whole{\mathbf V}
\newcommand{\modulo}[2]{[#1]_{#2}}
\def \essinf{\mathop{\rm essinf}}
\def\scriptf{{\mathcal F}}
\def\scriptg{{\mathcal G}}
\def\scriptm{{\mathcal M}}
\def\scriptb{{\mathcal B}}
\def\scriptc{{\mathcal C}}
\def\scriptt{{\mathcal T}}
\def\scripti{{\mathcal I}}
\def\scripte{{\mathcal E}}
\def\scriptv{{\mathcal V}}
\def\scriptw{{\mathcal W}}
\def\scriptu{{\mathcal U}}
\def\scriptS{{\mathcal S}}
\def\scripta{{\mathcal A}}
\def\scriptr{{\mathcal R}}
\def\scripto{{\mathcal O}}
\def\scripth{{\mathcal H}}
\def\scriptd{{\mathcal D}}
\def\scriptl{{\mathcal L}}
\def\scriptn{{\mathcal N}}
\def\scriptp{{\mathcal P}}
\def\scriptk{{\mathcal K}}
\def\frakv{{\mathfrak V}}
\def\C{\mathbb{C}}
\def\R{\mathbb{R}}
\def\Rn{{\mathbb{R}^n}}
\def\Sn{{{S}^{n-1}}}
\def\M{\mathcal{M}}
\def\N{\mathcal{N}}
\def\Q{{\mathbb{Q}}}
\def\Z{\mathbb{Z}}
\def\I{\mathcal{I}}
\def\D{\mathcal{D}}
\def\S{\mathcal{S}}
\def\supp{\operatorname{supp}}
\def\dist{\operatorname{dist}}
\def\essi{\operatornamewithlimits{ess\,inf}}
\def\esss{\operatornamewithlimits{ess\,sup}}

\author{Mingming Cao}
\address{Mingming Cao \\
         School of Mathematical Sciences \\
         Beijing Normal University \\
         Laboratory of Mathematics and Complex Systems \\
         Ministry of Education \\
         Beijing 100875 \\
         People's Republic of China}
\email{m.cao@mail.bnu.edu.cn}

\author{Qingying Xue}
\address{Qingying Xue\\
        School of Mathematical Sciences\\
        Beijing Normal University \\
        Laboratory of Mathematics and Complex Systems\\
        Ministry of Education\\
        Beijing 100875\\
        People's Republic of China}
\email{qyxue@bnu.edu.cn}

\thanks{The second author was supported partly by NSFC
(No. 11471041), the Fundamental Research Funds for the Central Universities (No. 2014KJJCA10) and NCET-13-0065.}


\date{April 22, 2016.}
\title[Fractional Integral Operator]{A Revisit on Commutators of linear and bilinear Fractional Integral Operator}
\maketitle

\begin{abstract}
Let $I_{\alpha}$ be the linear and $\I_{\alpha}$ be the bilinear fractional integral operators. In the linear setting, it is known that the two-weight inequality holds for the first order commutators of $I_{\alpha}$. But the method can't be used to obtain the two weighted norm inequality for the higher order commutators of $I_{\alpha}$. In this paper, we first give an alternative proof for the first order commutators of $I_{\alpha}$. This new approach allows us to consider the higher order commutators. This was done by showing that the commutator $[b,I_{\alpha}]$ can be represented as a finite linear combination of some paraproducts. Then, by using the Cauchy integral theorem, we show that the two-weight inequality holds for the higher order commutators of $I_{\alpha}$. In the bilinear setting, we present a dyadic proof for the characterization between $BMO$ and the boundedness of $[b,\I_{\alpha}]$. Moreover, some bilinear
paraproducts are also treated in order to obtain the boundedness of $[b,\I_{\alpha}]$.
\end{abstract}

\section{Introduction}
It is well known that the bilinear fractional integral operator $\mathcal{I}_{\alpha}$ is defined by
\begin{equation*}
\I_{\alpha}(f_1,f_2)(x) = \int_{\Rn} \int_{\Rn} \frac{f_1(x - y_1) f_2(x-y_2)}{(|y_1|+|y_2|)^{2n - \alpha}} dy_1 dy_2,\ \ 0< \alpha < 2n,
\end{equation*}
Its dyadic model operator $\I_{\alpha}^{\D}$ is defined by
$$
\I_{\alpha}^{\D}(f_1,f_2)(x)
:= \sum_{Q \in \D} |Q|^{\frac{\alpha}{n}} \langle f_1 \rangle_Q \langle f_2 \rangle_Q \cdot \mathbbm{1}_Q(x),\ x \in \Rn,
$$
where $\D$ is a dyadic grid on $\Rn$. The corresponding linear operators are denoted by $I_{\alpha}$ and $I_{\alpha}^{\D}$ respectively.
In 1982, Chanillo \cite{Chan} first studied the commutators of $I_{\alpha}$ (or Riesz potential), and then he used it to characterize the $BMO$ space.
In 2007, Lacey \cite{L} reconsidered the boundedness of the commutators of $I_{\alpha}$ in one dimension. He showed that the commutator with Riesz potential is a sum of dyadic paraproducts. Moreover, the result was also extended to
the multi-parameter setting. The reason why the new method is valid lies in that $I_{\alpha}$ can be represented as the averaging of its
dyadic version $I_{\alpha}^{\D}$.

In two-weight setting, in 1985, Bloom \cite{B} considered the two-weight behavior for the commutators of Hilbert transform $H$. It was proved that for two weights $\mu,\lambda$ in $A_p$ class, the commutator $[b,H]$ is bounded from $L^p(\mu)$ to $L^p(\lambda)$ if and only if the function $b$ satisfies
$$
||b||_{BMO(\nu)}:=\sup_Q \frac{1}{\nu(Q)} \int_Q |b(x)-\langle b \rangle|dx < \infty,\quad\quad \nu=\mu^{\frac1p} \lambda^{-\frac1p}.
$$
It is worth pointing out that this characterization involves three independent objects $b,\mu,\lambda$.
Recently, Holmes, Lacey and Wick \cite{HLW-1} gave a modern proof of Bloom's theorem when $p=2$. In addition, Bloom's $BMO(\nu)$ space was redefined by equivalent formulations, which can be used to characterize the two-weight boundedness of certain dyadic paraproducts. Subsequently, the general ideas in \cite{HLW-1} were applied to Calder\'{o}n-Zygmund operator $T$. By using the dyadic representation given by Hyt\"{o}nen in \cite{H}, the authors \cite{HLW-2} cleverly reduced the commutator $[b,T]$ to $[b,\mathbb{S}^{i,j}]$ ($\mathbb{S}^{i,j}$ is the dyadic shift with parameters $(i,j)$), and thus the commutator of Riesz transform could be used to characterize the $BMO(\nu)$ space.

Recently, by using similar ideas as in \cite{HLW-2} and modern methods of dyadic analysis, Holmes and Wick \cite{HW} obtained the upper bounds for higher order commutators of Calder\'{o}n-Zygmund operators. It is worthy to point out that some new techniques sponsored by authors in \cite{HW} can be well used to deal with the very difficult part, that is, the remainder term. Later on, using the classical Cauchy integral argument, Hyt\"{o}nen \cite{H-higher} provided a brief and totally new proof for the result in \cite{HW}. Still more recently, the authors in \cite{HRS} studied the first order commutators of $I_{\alpha}$ and they showed that the norm
$\big\|[b,I_{\alpha}]\big\|_{L^p(\mu^p) \rightarrow L^q(\lambda^q)}$ is equivalent to the norm $||b||_{BMO(\nu)}$ for $\mu,\lambda \in A_{p,q}$ and $\nu=\mu \lambda^{-1}$. When it comes to the higher order commutators of $I_{\alpha}$, we need to seek a new approach since the techniques used in \cite{HRS}
and \cite{HW} are invalid. The main reason lies in that, unlike the dyadic shift $\mathbb{S}^{ij}$, it is not symmetric and pretty much more complicated when dealing with the remainder terms.
More works for the multi-parameter type commutators can be found in \cite{DO}, \cite{FL} and \cite{LPPW}.

Motivated by the above works, in this paper, our objects of investigation are the commutators of the linear and bilinear fractional integral operators.
First, we will show that the two-weight inequality holds for the higher order commutators of $I_{\alpha}$, defined by
\begin{equation*}
C_b^k(I_{\alpha})(f)(x) = \int_{\Rn} \frac{(b(x)-b(y))^kf( y) }{|x-y|^{n - \alpha}} dy,\ \ 0< \alpha < n.
\end{equation*}
For the first order commutators ($k=1$), we will give a new proof by means of dyadic fractional integral operator $I_{\alpha}^{\D}$. The good thing is that, in this case, the commutator $[b,I_{\alpha}^{\D}]$ can be written as a sum of paraproducts. 
Fortunately, the Cauchy integral theorem provides us an efficient way by the reason that it enable us to reduce the proof to the first order case.

The main result of this paper is:
\begin{theorem}\label{Theorem-Higher}
Let $ 0 < \alpha < n $, $ 1 < p < q < \infty $ with $\frac{1}{q} = \frac{1}{p} - \frac{\alpha}{n}$. Suppose that weights $\mu, \lambda \in A_{p,q}$ and the function
$b \in BMO \cap BMO(\nu)$ with $\nu=\mu \lambda^{-1}$. Then, for any $k \geq 1$, it holds that
$$
\big\| C_b^k(I_{\alpha}) \big\|_{L^p(\mu^p) \rightarrow L^q(\lambda^q)}
\leq C_{n,p,q,k}([\mu]_{A_{p,q}}, [\lambda]_{A_{p,q}}) ||b||_{BMO(\nu)} ||b||_{BMO}^{k-1},
$$
where $C_{n,p,q,k} (\cdot, \cdot)$ is monotone increasing in both variables.
\end{theorem}

In the bilinear setting, the bilinear commutators will be of the following forms
$$
[b,T]_1(f_1,f_2)(x):= b T(f_1,f_2)(x) - T(b f_1,f_2)(x),
$$
and
$$
[b,T]_2(f_1,f_2)(x):= b T(f_1,f_2)(x) - T(f_1,b f_2)(x).
$$
We present a dyadic proof for the following characterization between dyadic $BMO$ and the commutator of $\I_{\alpha}$. The continuous version was proved in \cite{C}.
\begin{theorem}\label{Theorem-Bilinear}
Let $ 0 < \alpha < 2n $, $ 1 < p_1, p_2 < \infty $, and $\frac{1}{q} = \frac{1}{p_1} + \frac{1}{p_2} - \frac{\alpha}{n}$. Suppose that $b \in L_{loc}^1(\Rn)$. Then
$$
\big\| [b,\I_\alpha^{\D}]_j \big\|_{L^{p_1}(\Rn) \times L^{p_2}(\Rn) \rightarrow L^q(\Rn)}
\simeq ||b||_{BMO_{\D}},\ \ j=1,2.
$$\end{theorem}

\section{Preliminaries}
In order to show our results, we here present some definitions and lemmas, which will be needed later.
\subsection{Haar functions}
Let $h_Q^{\epsilon}$ be an $L^2$ normalized Haar function related to $Q \in \D$, where $\D$ is a dyadic grid on $\Rn$, and $\epsilon=(\epsilon_1,\ldots,\epsilon_n) \in \{0, 1\}^n$. With this we mean that $h_Q^\epsilon$, $Q = I_1 \times \cdots \times I_n$, is one of the $2^n$ functions $h_Q^\epsilon$, defined by
$$
 h_Q^\epsilon = h_{I_1}^{\epsilon_1} \otimes \cdots \otimes h_{I_n}^{\epsilon_n},
$$
where the one-dimension Haar functions :
$$
h_{I}^0 := |I|^{-\frac12}(\mathbbm{1}_{I_{-}} - \mathbbm{1}_{I_{+}}),\
h_{I}^1 := |I|^{-\frac12} \mathbbm{1}_{I}.
$$
Here $I_{-}$ and $I_{+}$ are the left and right halves of the interval $I$ respectively. For convenience, we write $\epsilon = 1$ when $\epsilon_i=1$ for all $i$. If $\epsilon \neq 1$, the Haar function is cancellative : $\int_\Rn h_Q^{\epsilon} = 0$. All the cancellative Haar functions on a fixed dyadic grid $\D$ form an orthonormal basis of $L^2(\Rn)$. If $f \in L^2(\Rn)$, we may thus write
\begin{equation}\label{decomposition}
f = \sum_{Q \in \D,\epsilon \neq 1} \langle f,h_Q^{\epsilon} \rangle h_Q^{\epsilon},
\end{equation}
where $\langle \cdot,\cdot \rangle$ denotes the usual inner product on $L^2(\Rn)$. Moreover, we have the basic facts as follows:
\begin{equation}\label{average}
\langle f \rangle_Q = \sum_{\substack{P \in \D:P \supsetneq Q \\ \epsilon \neq 1}} \langle f,h_P^{\epsilon} \rangle h_P^{\epsilon}(Q).
\end{equation}
And there holds that
$$
h_Q^{\epsilon} h_Q^{\eta} = |Q|^{-\frac12}h_Q^{\epsilon + \eta},
$$
where $\epsilon + \eta$ is defined by
\begin{equation*}
(\epsilon + \eta)_i : = \delta_{(\epsilon_i,\eta_i)}=
\begin{cases}
0, \ &\text{if } \epsilon_i \neq \eta_i \\
1, \ &\text{if } \epsilon_i = \eta_i .
\end{cases}
\end{equation*}
\subsection{Multiple $A_{(\vec{p},q)}$ weights}
Suppose that $\vec{\omega} = (\omega_1, \cdots, \omega_m)$ and each $\omega_i$ $(i=1,\cdots,m)$ is a nonnegative function on $\Rn$.
We say that $\vec{\omega}$ satisfies the $A_{({\vec{p}},q)}$ condition, written $\vec{\omega} \in A_{(\vec{p},q)}$, if it satisfies
$$
\sup_Q \left( \frac{1}{|Q|} \int_Q \nu_{\vec{\omega}}^q dx \right)^{1/q}
\prod_{i=1}^m \left( \frac{1}{|Q|} \int_Q \omega_i^{-p_i'} dx \right)^{1/{p_i'}} < \infty,
$$
where $\nu_{\vec{\omega}} = \prod_{i=1}^m \omega_i$. If $p_i = 1$, then $(\frac{1}{Q} \int_Q \omega_i^{-p_i'})^{1/{p_i'}}$ is understood as $(\inf_Q \omega_i)^{-1}$.

The multilinear fractional type maximal operator $\mathcal{M}_\alpha$, which is defined by
$$
\mathcal{M}_\alpha(\vec{f})(x) = \sup_{Q\ni x} \prod_{i=1}^m \frac{1}{|Q|^{1-\frac{\alpha}{mn}}} \int_Q |f_i(y_i)| \;dy_i,   \quad\quad \text{for} \ \ 0< \alpha < mn
$$
where the supremum is taken over all cubes $Q$ containing $x$ in $\Rn$ with the sides parallel to the axes.
In addition,

We summarize some known results of $\mathcal{M}_\alpha $ and $\mathcal{I}_\alpha $ \cite{CX} as follows.
\begin{theorem}\label{M-alpha-I-alpha}
Let $ 0 \leq \alpha < mn $, $ 1< p_1, \cdots, p_m < \infty $, $ \frac{1}{p} = \frac{1}{p_{1}} + \cdots + \frac{1}{p_{m}} $ and
$\frac{1}{q} = \frac{1}{p} - \frac{\alpha}{n}$. Then, $ \vec{\omega} \in A_{(\vec{p},q)} $ if and only if either of the following two inequalities hold:
\begin{eqnarray}
\label{M-alpha}{\big\| \mathcal{M}_\alpha (\vec{f}) \big\|}_{L^q({\nu_{\vec{\omega} }}^q)}
& \lesssim & \prod_{i = 1}^m{\big\|f_i\big\|}_{L^{p_i}({\omega_i}^{p_i})}; \\
\label{I-alpha}{\big\| \mathcal{I}_\alpha (\vec{f}) \big\|}_{L^q({\nu_{\vec{\omega} }}^q)}
& \lesssim & \prod_{i = 1}^m {\big\| f_i \big\|}_{L^{p_i}({\omega_i}^{p_i})}.
\end{eqnarray}
\end{theorem}

For a dyadic grid $\D$, the dyadic square function is defined by
$$
S_{\D}f := \left(\sum_{Q \in \D,\epsilon \neq 1} \big| \langle f,h_Q^{\epsilon} \rangle\big|^2 \frac{\mathbbm{1}_Q}{|Q|}\right)^{\frac12}.
$$
\begin{lemma}(\cite{CM})\label{SD-f}
If $w \in A_p$, then there holds that
$$
\big\|S_{\D}f\big\|_{L^p(w)}
\lesssim [w]_{A_p}^{\max\{\frac12,\frac{1}{p-1}\}} ||f||_{L^p(w)}.
$$
\end{lemma}
\subsection{Weighted $BMO$}
Let $w$ be a weight on $\Rn$. The weighted $BMO$ space $BMO(w)$ is defined to be the space of all locally integrable functions $b$ satisfying
$$
||b||_{BMO(w)}:=\sup_{Q: \text{ cubes in } \Rn}\frac{1}{w(Q)}\int_{Q}|b(x)-\langle b \rangle_Q|dx < \infty.
$$
If $w=1$, we simply denote it by $BMO$. If $w \in A_p$, there holds
\begin{equation}\label{BMO-BMO}
||b||_{BMO} \leq ||b||_{BMO^r(w)} \leq c_{n,p,[w]_{A_p}} ||b||_{BMO(w)},
\end{equation}
where $w'=w^{1-p'}$ and
$$
||b||_{BMO^r(w)}^r := \sup_{Q: \text{ cubes in } \Rn}\frac{1}{w(Q)}\int_{Q}|b(x)-\langle b \rangle_Q|^r dw'.
$$
Given a dyadic grid $\D$, we define the dyadic versions of these spaces, $BMO_{\D}(w)$ and $BMO_{\D}^r(w)$, by taking the supremum over $Q \in \D$.
The inequality $(\ref{BMO-BMO})$ in the dyadic setting also holds. Moreover, using the equations $(\ref{decomposition})$ and $(\ref{average})$, we obtain the following form of dyadic $BMO$ :
$$
||b||_{BMO_{\D}^2}=\sup_{J \in \D}\bigg(\frac{1}{|J|} \sum_{\substack{I:I \subset J \\ \epsilon \neq 1}} |\langle b,h_I^{\epsilon} \rangle|^2 \bigg)^{1/2}.
$$

The following inequality proved in \cite{HLW-2} will be very useful for our proofs.
\begin{lemma}\label{b-f-SD}
If $w \in A_2$, then there holds that
$$
|\langle b,f \rangle|
\lesssim [w]_{A_2} ||b||_{BMO_{\D}^2(w)} ||S_{\D}f||_{L^1(w)}.
$$
\end{lemma}
\section{The commutator $[b,I_{\alpha}^{\D}]$ revisited}
In this section, we will demonstrate the first order case in Theorem $\ref{Theorem-Higher}$. To finish this work, we first need the fact that
$I_{\alpha}$ can be recovered from $I_{\alpha}^{\D}$, which was shown in \cite{HRS}, and essentially goes back to \cite{H} and \cite{PTV}.
Thus, it is enough to prove the result holds for $[b,I_{\alpha}^{\D}]$.

We introduce the following paraproduct operators.
\begin{eqnarray*}
\Pi_b f &:=& \sum_{Q \in \D, \epsilon \neq 1} \langle b,h_{Q}^{\epsilon} \rangle \langle f \rangle_Q  h_{Q}^{\epsilon}, \\
\Gamma_b f &:=& \sum_{Q \in \D} \sum_{\substack{\epsilon,\eta \neq 1 \\ \epsilon \neq \eta}} \langle b, h_Q^{\epsilon} \rangle \langle f,h_{Q}^{\eta} \rangle \frac{1}{\sqrt{|Q|}} h_Q^{\epsilon+\eta}, \\
B_k(b,f) &:=& \sum_{Q \in \D}  \sum_{\epsilon,\eta \neq 1} \langle b, h_Q^{\epsilon} \rangle \langle f,h_{Q^{(k)}}^{\eta} \rangle h_{Q^{(k)}}^{\eta} h_Q^{\epsilon}.
\end{eqnarray*}
where $Q^{(k)}$ denotes the $k$-th dyadic ancestor of $Q$. Let $\Pi_b^*$ be the adjoint of $\Pi_b$ in $L^2(\Rn)$, then
$$
\Pi_b^* f = \sum_{Q \in \D, \epsilon \neq 1} \langle b,h_{Q}^{\epsilon} \rangle \langle f,h_{Q}^{\epsilon} \rangle \frac{\mathbbm{1}_Q}{|Q|}.
$$
Moreover, it is easy to see that $B_0(b,f) = \Pi_b^*f + \Gamma_b f$. We need the following lemma:
\begin{lemma}\label{combination}
The commutator $[b,I_{\alpha}^{\D}]$ is a linear combination of the four terms :
$$
\Pi_b \circ I_{\alpha}^{\D},\ I_{\alpha}^{\D} \circ \Pi_b^*,\ \Pi_b^* \circ I_{\alpha}^{\D},\
\sum_{k=1}^{\infty} 2^{-k \alpha/n} B_k(b, \cdot) \circ I_{\alpha}^{\D}.
$$
\end{lemma}
\begin{proof}
Using the Haar expansion $(\ref{decomposition})$, we can write
$$
[b,I_\alpha^{\D}] (f)
=\sum_{I,J \in \D} \sum_{\epsilon,\eta \neq 1} \langle b, h_{J}^{\eta} \rangle  \langle f, h_{I}^{\epsilon} \rangle [h_{J}^{\eta},I_\alpha^{\D}] h_{I}^{\epsilon}.
$$
Note that $\epsilon \neq 1$, it holds that
\begin{equation}\label{I-h}
I_{\alpha}^{\D}(h_{I}^{\epsilon})
=\sum_{Q \in \D : Q \subsetneq I} |Q|^{\frac{\alpha}{n}} \langle h_{I}^{\epsilon} \rangle_Q \mathbbm{1}_Q
=\Big( \sum_{Q \in \D : Q \subsetneq I } |Q|^{\frac{\alpha}{n}} \mathbbm{1}_Q \Big) h_{I}^{\epsilon}
=c_\alpha |I|^{\frac{\alpha}{n}} h_{I}^{\epsilon}.
\end{equation}
Thus, it yields that
\begin{align*}
[h_{J}^{\eta},I_\alpha^{\D}] h_{I}^{\epsilon}
=\begin{cases}
c_\alpha |I|^{\frac{\alpha}{n}} \frac{\mathbbm{1}_{I}}{|I|} - I_{\alpha}^{\D}\big(\frac{\mathbbm{1}_{I}}{|I|}\big),\ \ &\text{if}\ I=J \text{ and } \epsilon = \eta;\\
c_\alpha h_{I}^{\epsilon}(J) \big(|I|^{\frac{\alpha}{n}} - |J|^{\frac{\alpha}{n}}\big)  h_J^{\eta} ,\ \ &\text{if}\ J \subsetneq I;\\
0,\ \ & \text{otherwise}.
\end{cases}
\end{align*}
Moreover, by $(\ref{I-h})$, we have
\begin{equation}\label{I-f-h}
\langle I_{\alpha}^{\D} f, h_{I}^{\epsilon} \rangle
= \langle  f, I_{\alpha}^{\D} h_{I}^{\epsilon} \rangle
=c_{\alpha} |I|^{\frac{\alpha}{n}} \langle  f, h_{I}^{\epsilon} \rangle.
\end{equation}
Therefore,
$$
c_{\alpha} \sum_{I \in \D, \epsilon \neq 1} \langle b, h_I^{\epsilon} \rangle \langle f,h_I^{\epsilon} \rangle |I|^{\frac{\alpha}{n}} \frac{\mathbbm{1}_I}{|I|}
=\sum_{I \in \D, \epsilon \neq 1} \langle b, h_I^{\epsilon} \rangle \langle I_{\alpha}^{\D}f,h_I^{\epsilon} \rangle \frac{\mathbbm{1}_I}{|I|}
= \Pi_b^* \circ I_{\alpha}^{\D}(f).
$$
From the averaging formula $(\ref{average})$ and $(\ref{I-f-h})$, it follows that
\begin{equation}\label{I-f-J}
\langle I_{\alpha}^{\D} f \rangle_J
=\sum_{\substack{I: I \supsetneq J \\ \epsilon \neq 1}} \langle I_{\alpha}^{\D} f,  h_{I}^{\epsilon} \rangle  h_{I}^{\epsilon}(J)
=c_{\alpha} \sum_{\substack{I: I \supsetneq J \\ \epsilon \neq 1}} |I|^{\frac{\alpha}{n}} \langle  f, h_{I}^{\epsilon} \rangle h_{I}^{\epsilon}(J).
\end{equation}
Then, it holds that
$$
\sum_{I \in \D, \epsilon \neq 1} \langle b, h_I^{\epsilon} \rangle \langle f,h_I^{\epsilon} \rangle I_{\alpha}^{\D}\Big(\frac{\mathbbm{1}_I}{|I|}\Big)
= I_{\alpha}^{\D} \circ \Pi_b^* (f).
$$
The equation $(\ref{I-f-J})$ gives that
$$
c_{\alpha} \sum_{J \in \D,\eta \neq 1} \sum_{\substack{I: J \subsetneq I \\ \epsilon \neq 1}} \langle b, h_{J}^{\eta} \rangle  \langle f, h_{I}^{\epsilon} \rangle h_{I}^{\epsilon}(J)|I|^{\frac{\alpha}{n}} h_J^{\eta}
=\sum_{J \in \D,\eta \neq 1} \langle b, h_{J}^{\eta} \rangle \langle I_{\alpha}^{\D}f \rangle_J h_J^{\eta}
=\Pi_b \circ I_{\alpha}^{\D}(f).
$$
Now we consider the contribution of the last term. Using $(\ref{I-f-h})$, we get
\begin{align*}
&c_{\alpha} \sum_{\substack{I,J \in \D \\ J \subsetneq I}} \sum_{\epsilon,\eta \neq 1} \langle b, h_{J}^{\eta} \rangle  \langle f, h_{I}^{\epsilon} \rangle h_{I}^{\epsilon}(J) |J|^{\frac{\alpha}{n}} h_J^{\eta} \\
&=\sum_{k=1}^{\infty} 2^{-k\alpha/n} \sum_{I \in \D} \sum_{\epsilon,\eta \neq 1} \langle b, h_{J}^{\eta} \rangle \Big(c_{\alpha} \langle f, h_{J^{(k)}}^{\epsilon} \rangle |J^{(k)}|^{\frac{\alpha}{n}}\Big) h_{J^{(k)}}^{\epsilon}(J)  h_J^{\eta} \\
&=\sum_{k=1}^{\infty} 2^{-k\alpha/n} \sum_{I \in \D} \sum_{\epsilon,\eta \neq 1} \langle b, h_{J}^{\eta} \rangle \langle I_{\alpha}^{\D}f,h_{J^{(k)}}^{\epsilon} \rangle h_{J^{(k)}}^{\epsilon}(J)  h_J^{\eta} \\
&=\sum_{k=1}^{\infty} 2^{-k\alpha/n} B_k(b, \cdot) \circ I_{\alpha}^{\D}(f).
\end{align*}
This completes the proof of Lemma \ref{combination}.
\end{proof}
The two-weight inequalities for paraproducts are as follows.
\begin{lemma}\label{D-b-k}
Let $k \geq 0$, $ 1 < p < \infty $, and $\mu,\lambda \in A_p$. Let $\nu=\mu^{\frac1p} \lambda^{-\frac1p}$. Then
\begin{eqnarray}
\big\| T_b(f)  \big\|_{L^{p}(\lambda)} & \lesssim & ||b||_{BMO_{\D}^2(\nu)} ||f||_{L^p(\mu)}, \\
\big\| B_k(f,b)  \big\|_{L^{p}(\lambda)} & \lesssim & 2^{nk/2} ||b||_{BMO_{\D}^2(\nu)} ||f||_{L^p(\mu)},
\end{eqnarray}
where $T_b$ denotes either one of the operators $\Pi_b$, $\Pi_b^*$, $\Gamma_b$ and $B_k(b, \cdot)$. The implied constant is independent of $k$.
\end{lemma}
\begin{proof}
(i) Let $ \big\| g \big\|_{L^{p'}(\lambda^{-p'/p})} =1$. For $k \geq 1$, we will consider,
$$
\langle B_k(b,f), g \rangle
= \Big\langle b, \sum_{I \in \D} \sum_{\epsilon,\eta \neq 1} \langle f,h_{I^{(k)}}^{\eta} \rangle h_{I^{(k)}}^{\eta}(I)
\langle g, h_I^{\epsilon} \rangle h_I^{\epsilon} \Big\rangle
:= \langle b,\Phi \rangle.
$$
Then we have
$$
S_{\D} \Phi = \bigg( \sum_{I \in \D,\epsilon \neq 1} \Big| \sum_{\eta \neq 1}\langle f,h_{I^{(k)}}^{\eta} \rangle h_{I^{(k)}}^{\eta}(I) \Big|^2
|\langle g, h_I^{\epsilon} \rangle|^2 \frac{\mathbbm{1}_I}{|I|} \bigg)^{1/2}
\lesssim Mf \cdot S_{\D}g.
$$
Thus, by the H\"{o}lder inequality and Lemma $\ref{SD-f}$, it follows that
$$
\big\| S_{\D} \Phi \big\|_{L^1(\nu)}
\lesssim \big\| Mf \big\|_{L^p(\mu)} \big\| S_{\D} g \big\|_{L^{p'}(\lambda^{-p'/p})}
\lesssim \big\| f \big\|_{L^p(\mu)} \big\| g \big\|_{L^{p'}(\lambda^{-p'/p})}.
$$
By Lemma $\ref{b-f-SD}$, we obtain that
$$
\big\| B_k(b,f)  \big\|_{L^{p}(\lambda)} \lesssim ||b||_{BMO_{\D}^2(\nu)} ||f||_{L^p(\mu)},\ k \geq 1.
$$
Using the similar arguments, we can show that the same result still holds for $\Pi_b$ and $\Gamma_b$. A simple duality argument gives the corresponding result for $\Pi_b^*$.
Finally, the two-weight inequality for $B_0(b,\cdot)$ follows from the fact $B_0(b,\cdot) = \Pi_b^* + \Gamma_b$.

(ii) We next show the two-weight boundedness of $B_k(\cdot,b)$. It suffices to show the cake $k \geq 1$, since $B_0(f,b)=B_0(b,f)$.
For any $ g \in L^{p'}(\lambda^{1-p'})$, we have
$$
\langle B_k(f,b), g \rangle
= \Big\langle b, \sum_{I \in \D} \sum_{\epsilon,\eta \neq 1} \langle f,h_{I}^{\epsilon} \rangle \langle g, h_I^{\epsilon} \rangle h_{I^{(k)}}^{\eta}(I) h_{I^{(k)}}^{\eta} \Big\rangle
:= \langle b,\Psi \rangle.
$$
By Lemma $\ref{b-f-SD}$, it is enough to prove
$$
\big\| S_{\D} \Psi \big\|_{L^1(\nu)}
\lesssim \big\| f \big\|_{L^p(\mu)} \big\| g \big\|_{L^{p'}(\lambda^{1-p'})}.
$$
To obtain this, we write
$$
S_{\D}\Psi = \bigg( \sum_{J \in \D,\eta \neq 1} \Big(\sum_{\substack{I:I^{(k)}=J \\ \epsilon \neq 1}} \langle f,h_{I}^{\epsilon} \rangle \langle g, h_I^{\epsilon} \rangle h_{I^{(k)}}^{\eta}(I) \Big)^2 \frac{\mathbbm{1}_J}{|J|} \bigg)^{1/2}.
$$
Then, the H\"{o}lder inequality implies that
\begin{align*}
S_{\D}\Psi
& \lesssim \sum_{J \in \D} \Big(\sum_{\substack{I:I^{(k)}=J \\ \epsilon \neq 1}} |\langle f,h_{I}^{\epsilon}\rangle| |\langle g, h_I^{\epsilon} \rangle| \Big) \frac{\mathbbm{1}_J}{|J|} \\
& \leq \sum_{J \in \D} \Big(\sum_{\substack{I:I^{(k)}=J \\ \epsilon \neq 1}} |\langle f,h_{I}^{\epsilon} \rangle|^2 \Big)^{1/2}
\Big(\sum_{\substack{I:I^{(k)}=J \\ \epsilon \neq 1}} |\langle g, h_I^{\epsilon} \rangle|^2 \Big)^{1/2} \frac{\mathbbm{1}_J}{|J|} \\
& \leq \bigg(\sum_{J \in \D} \sum_{\substack{I:I^{(k)}=J \\ \epsilon \neq 1}} |\langle f,h_{I}^{\epsilon} \rangle|^2 \frac{\mathbbm{1}_J}{|J|} \bigg)^{1/2}
\bigg(\sum_{J \in \D} \sum_{\substack{I:I^{(k)}=J \\ \epsilon \neq 1}} |\langle g, h_I^{\epsilon} \rangle|^2 \frac{\mathbbm{1}_J}{|J|} \bigg)^{1/2} \\
&:=(S_k f) (S_k g).
\end{align*}
Thus, we are reduced to demonstrating for $w \in A_p$
$$
\big\| S_k f \big\|_{L^p(w)} \lesssim 2^{nk/2}||f||_{L^p(w)}.
$$
In light of extrapolation theorem, it suffices to show it for the case $p=2$.
Indeed, there holds that
\begin{align*}
\big\| S_k f \big\|_{L^2(w)}^2
&=\sum_{J \in \D} \sum_{\substack{I:I^{(k)}=J \\ \epsilon \neq 1}} |\langle f, h_I^{\epsilon} \rangle|^2 \langle w^{-1} \rangle_I^{-1}
\Big(\langle w^{-1} \rangle_I \langle w \rangle_J \Big) \\
&\leq 2^{nk} [w]_{A_2} \sum_{I \in \D, \epsilon \neq 1} |\langle f, h_I^{\epsilon} \rangle|^2 \langle w^{-1} \rangle_I^{-1} \\
&\leq 2^{nk} [w]_{A_2} \sum_{I \in \D, \epsilon \neq 1} |\langle f, h_I^{\epsilon} \rangle|^2 \langle w \rangle_I \\
&\leq 2^{nk} [w]_{A_2} \big\| S_{\D} f \big\|_{L^2(w)}^2
\lesssim  2^{nk} \big\| f \big\|_{L^2(w)}^2 .
\end{align*}
This completes the proof.
\end{proof}

The case $k=1$ in Theorem $\ref{Theorem-Higher}$ follows from Lemma $\ref{combination}$, Lemma $\ref{D-b-k}$ and Lemma $\ref{M-alpha-I-alpha}$.

\section{Higher order commutator}
Denote $T_{\alpha}=C_b^1(I_{\alpha})$, and $F(z) = e^{bz} T_{\alpha} e^{-bz}$.
Then it is easy to see that
$$
C_b^{k+1}(I_{\alpha}) = C_b^k(T_{\alpha}) = F^{(k)}(0) = \frac{k!}{2 \pi i} \oint_{C} \frac{F(z)}{z^{k+1}} dz,
$$
where the integral is over any closed path around the origin.
If we set
$$
\Phi(z) := C_{n,p,q}([e^{bz} \mu]_{A_{p,q}}, [e^{bz} \lambda]_{A_{p,q}}),
$$
then, for any $r > 0$, we have
\begin{align*}
&\big\| C_b^{k+1}(I_{\alpha}) \big\|_{L^p(\mu^p) \rightarrow L^q(\lambda^q)} \\
&\leq \frac{k!}{2 \pi} \oint_{|z|=r} \big\| e^{bz} T_{\alpha} e^{-bz} \big\|_{L^p(\mu^p) \rightarrow L^q(\lambda^q)} \frac{|dz|}{|z|^{k+1}} \\
&\leq  \frac{k!}{2 \pi} \oint_{|z|=r} \big\| T_{\alpha} \big\|_{L^p((e^{bz} \mu)^p) \rightarrow L^q((e^{bz} \lambda)^q)} \frac{|dz|}{|z|^{k+1}} \\
&\leq k! r^{-k} ||b||_{BMO(\nu)} \sup_{|z| \leq r} \big| \Phi(z) \big|.
\end{align*}
We need the following relationship between the $A_{p,q}$ weights and the $BMO$ space.
\begin{lemma}\label{A(p,q)-BMO}
Let $1 < p,q < \infty$, $w \in A_{p,q}$ and $b \in BMO$. Then, there are constants $c_{n,p,q}, c_{n,p,q}' > 0$, depending only on the indicated parameters, such that
$$
[e^{Re(bz)} w]_{A_{p,q}} \leq c_{n,p,q}' [w]_{A_{p,q}}
$$
for all $z \in \C$ with
$$
|z| \leq \frac{c_{n,p,q}}{||b||_{BMO} (w^q)_{A_{q_0}}}, \ \ q_0=1+q/{p'}.
$$
\end{lemma}
The notation $(w)_{A_p}$ means that $(w)_{A_p} := \max \big\{[w]_{A_{\infty}},\ [w^{1-p'}]_{A_{\infty}}\big\}$.

Note that $w \in A_{p,q}$ if and only if $w^q \in A_{q_0}$, and $[w]_{A_{p,q}} = [w^q]_{A_{q_0}}$. Thus, the above lemma follows from this basic fact and Lemma $2.1$ \cite{H-higher}.

We continue our proof. If we take
$$
r=\frac{c_{n,p,q}}{||b||_{BMO} \max \big\{ (\mu^q)_{A_{q_0}}, (\lambda^q)_{A_{q_0}} \big\}},
$$
then Lemma $\ref{A(p,q)-BMO}$ implies that
$$
\sup_{|z| \leq r} \big| \Phi(z) \big|
\leq C_{n,p,q}(c_{n,p,q}' [\mu]_{A_{p,q}}, c_{n,p,q}' [\lambda]_{A_{p,q}})
:=C_{n,p,q}'([\mu]_{A_{p,q}}, [\lambda]_{A_{p,q}}).
$$
Collecting the above estimates, we deduce that
$$
\big\| C_b^{k+1}(I_{\alpha}) \big\|_{L^p(\mu^p) \rightarrow L^q(\lambda^q)}
\leq ||b||_{BMO(\nu)}  ||b||_{BMO}^{k} C_{n,p,q,k}([\mu]_{A_{p,q}}, [\lambda]_{A_{p,q}}) .
$$
\qed
\section{Bilinear Paraproduct Operators}
In this section, we will treat the boundedness of four bilinear paraproduct operators, which are defined by
\begin{eqnarray*}
\Lambda_b(f_1,f_2)
&:=&\sum_{\substack{P,Q_1,Q_2 \in \D \\ P \subsetneq Q_1}} \sum_{\eta,\epsilon_1,\epsilon_2 \neq 1} \langle b,h_P^{\eta} \rangle \langle f_1, h_{Q_1}^{\epsilon_1} \rangle
\langle f_2,h_{Q_2}^{\epsilon_2} \rangle |Q_1 \cap Q_2|^{\frac{\alpha}{n}} h_{Q_1}^{\epsilon_1}(P) h_P^{\eta} h_{Q_2}^{\epsilon_2};\\
\Delta_b(f_1,f_2)
&:=&\sum_{\substack{P,Q_1,Q_2 \in \D \\ P \subsetneq Q_1}} \sum_{\eta,\epsilon_1,\epsilon_2 \neq 1} \langle b,h_P^{\eta} \rangle \langle f_1, h_{Q_1}^{\epsilon_1} \rangle
\langle f_2,h_{Q_2}^{\epsilon_2} \rangle |P \cap Q_2|^{\frac{\alpha}{n}} h_P^{\eta} h_{Q_1}^{\epsilon_1}  h_{Q_2}^{\epsilon_2};\\
\Xi_b(f_1,f_2)
&:=&\sum_{\substack{Q_1,Q_2 \in \D \\ Q_1 \subsetneq Q_2}} \sum_{\epsilon_1,\epsilon_2 \neq 1} \langle b,h_{Q_1}^{\epsilon_1} \rangle \langle f_1, h_{Q_1}^{\epsilon_1} \rangle
\langle f_2,h_{Q_2}^{\epsilon_2} \rangle |Q_1|^{\frac{\alpha}{n}} h_{Q_2}^{\epsilon_2}(Q_1) \frac{\mathbbm{1}_{Q_1}}{|Q_1|};\\
\Theta_b(f_1,f_2)
&:=&\sum_{Q_1,Q_2 \in \D} \sum_{\epsilon_1,\epsilon_2 \neq 1} \langle b,h_{Q_1}^{\epsilon_1} \rangle \langle f_1, h_{Q_1}^{\epsilon_1} \rangle
\langle f_2,h_{Q_2}^{\epsilon_2} \rangle h_{Q_2}^{\epsilon_2}(Q)\Big(\sum_{Q:Q_1 \subsetneq Q \subsetneq Q_2} |Q|^{\frac{\alpha}{n}} \frac{\mathbbm{1}_Q}{|Q|}\Big).
\end{eqnarray*}
The two-weight results for the above bilinear paraproducts are as follows.
\begin{proposition}\label{Pro-Pra}
Let $b \in BMO^2_{\D}$ and $\D$ be a fixed dyadic grid on $\Rn$. Then
$$
\big\| T_b(f_1,f_2) \big\|_{L^q(\Rn)}
\lesssim ||b||_{BMO^2_{\D}} \big\| f_1 \big\|_{L^{p_1}(\Rn)} \big\| f_2 \big\|_{L^{p_2}(\Rn)},
$$
where $T_b$ denotes either one of the operators $\Lambda_b$, $\Delta_b$, $\Xi_b$ and $\Theta_b$.
\end{proposition}
The remainder of this section is devoted to the proof of Proposition $\ref{Pro-Pra}$.
\subsection{Estimate of $\Lambda_b$}
To analyze $\Lambda_b$. We perform the decomposition
\begin{align*}
\Lambda_b(f_1,f_2)
&=\sum_{\substack{P,Q_1,Q_2 \in \D \\ P \subsetneq Q_1 \subsetneq Q_2}} \sum_{\eta,\epsilon_1,\epsilon_2 \neq 1}
+\sum_{\substack{P,Q_1,Q_2 \in \D \\ P \subsetneq Q_1 = Q_2}} \sum_{\eta,\epsilon_1,\epsilon_2 \neq 1}
+\sum_{\substack{P,Q_1,Q_2 \in \D \\ P \subsetneq Q_1, Q_2 \subsetneq Q_1}} \sum_{\eta,\epsilon_1,\epsilon_2 \neq 1} \\
&:= \Lambda_b^{1}(f_1,f_2) + \Lambda_b^{2}(f_1,f_2) + \Lambda_b^{3}(f_1,f_2).
\end{align*}
We will dominate the three parts consecutively. Observe that
\begin{align*}
\Lambda_b^{1}(f_1,f_2)
&=\sum_{\substack{P,Q_1 \in \D \\ P \subsetneq Q_1}} \sum_{\eta,\epsilon_1 \neq 1} \langle b,h_P^{\eta} \rangle \langle f_1, h_{Q_1}^{\epsilon_1} \rangle |Q_1|^{\frac{\alpha}{n}}  h_{Q_1}^{\epsilon_1}(P) h_{P}^{\eta} \Big(\sum_{\substack{Q_2 \in \D:
Q_2 \supsetneq Q_1 \\ \epsilon_2 \neq 1}} \langle f_2, h_{Q_2}^{\epsilon_2} \rangle h_{Q_2}^{\epsilon_2}(Q_1)\Big)\\
&=\sum_{\substack{P,Q_1 \in \D \\ P \subsetneq Q_1}} \sum_{\eta,\epsilon_1 \neq 1} \langle b,h_P^{\eta} \rangle \langle f_1, h_{Q_1}^{\epsilon_1} \rangle  \langle f_2 \rangle_{Q_1} |Q_1|^{\frac{\alpha}{n}}  h_{Q_1}^{\epsilon_1}(P) h_{P}^{\eta}.
\end{align*}
Then, it follows that
\begin{equation}\label{Lambda-1}
\big\langle \Lambda_b^{1}(f_1,f_2),g \big\rangle
=\langle b, \Phi \rangle,
\end{equation}
where
$$
\Phi=\sum_{\substack{P,Q_1 \in \D \\ P \subsetneq Q_1}} \sum_{\eta,\epsilon_1 \neq 1} \langle g,h_P^{\eta} \rangle \langle f_1, h_{Q_1}^{\epsilon_1} \rangle  \langle f_2 \rangle_{Q_1} |Q_1|^{\frac{\alpha}{n}} h_{Q_1}^{\epsilon_1}(P) h_{P}^{\eta}.
$$
Notice that
\begin{equation}\aligned\label{I-alpha-f1-f2}
(S_{\D}\Phi)^2
&=\sum_{P \in \D,\eta \neq 1}|\langle g,h_P^{\eta} \rangle |^2 \Big(\sum_{\substack{Q_1:Q_1 \subsetneq P \\ \epsilon_1 \neq 1}}
\langle f_1, h_{Q_1}^{\epsilon_1} \rangle \langle f_2 \rangle_{Q_1}|Q_1|^{\frac{\alpha}{n}} h_{Q_1}^{\epsilon_1}(P)\Big)^2 \frac{\mathbbm{1}_P}{|P|} \\
&\lesssim \sum_{P \in \D,\eta \neq 1}|\langle g,h_P^{\eta} \rangle |^2 \Big(\sum_{Q_1 \in \D} |Q_1|^{\frac{\alpha}{n}}
\langle |f_1| \rangle_{Q_1} \langle |f_2| \rangle_{Q_1}\Big)^2 \frac{\mathbbm{1}_P}{|P|} \\
&\leq (S_{\D} g)^2 \cdot \I_{\alpha}(|f_1|,|f_2|)^2.
\endaligned
\end{equation}
Consequently, combining Lemma $\ref{b-f-SD}$, $(\ref{Lambda-1})$, $(\ref{I-alpha-f1-f2})$ with $(\ref{I-alpha})$, we get
\begin{align*}
\big| \big\langle \Lambda_b^{1}(f_1,f_2),g \big\rangle \big|
&\lesssim ||b||_{BMO^2_{\D}} \big\| S_{\D} \Phi \big\|_{L^1(\Rn)} \\
&\lesssim ||b||_{BMO^2_{\D}} \big\|S_{\D} g \cdot \I_{\alpha}(|f_1|,|f_2|)\big\|_{L^1(\Rn)} \\
&\lesssim ||b||_{BMO^2_{\D}} \big\|S_{\D} g \big\|_{L^{q'}(\Rn)} \big\| \I_{\alpha}(|f_1|,|f_2|)\big\|_{L^q(\Rn)} \\
&\lesssim ||b||_{BMO^2_{\D}} \big\| g \big\|_{L^{q'}(\Rn)} \big\| f_1 \big\|_{L^{p_1}(\Rn)} \big\| f_2 \big\|_{L^{p_2}(\Rn)}.
\end{align*}
Then, Riesz representation theorem gives that
$$
\big\| \Lambda_b^1(f_1,f_2) \big\|_{L^q(\Rn)}
\lesssim ||b||_{BMO^2_{\D}} \big\| f_1 \big\|_{L^{p_1}(\Rn)} \big\| f_2 \big\|_{L^{p_2}(\Rn)}.
$$

Next we consider the contribution of $\Lambda_b^{2}$.
\begin{align*}
\Lambda_b^{2}(f_1,f_2)
&=\sum_{\substack{P,Q_1 \in \D \\ P \subsetneq Q_1}} \sum_{\eta,\epsilon_1, \epsilon_2 \neq 1} \langle b,h_P^{\eta} \rangle \langle f_1, h_{Q_1}^{\epsilon_1} \rangle
\langle f_2,h_{Q_1}^{\epsilon_2} \rangle |Q_1|^{\frac{\alpha}{n}} h_{Q_1}^{\epsilon_1}(P) h_{Q_1}^{\epsilon_2}(P) h_{P}^{\eta}
\end{align*}
and
\begin{align*}
\big\langle \Lambda_b^{2}(f_1,f_2),g \big\rangle
=\langle b, \Psi \rangle,
\end{align*}
where
$$
\Psi=\sum_{\substack{P,Q_1 \in \D \\ P \subsetneq Q_1}} \sum_{\eta,\epsilon_1,\epsilon_2 \neq 1} \langle b,h_P^{\eta} \rangle \langle f_1, h_{Q_1}^{\epsilon_1} \rangle
\langle f_2,h_{Q_1}^{\epsilon_2} \rangle |Q_1|^{\frac{\alpha}{n}} h_{Q_1}^{\epsilon_1}(P) h_{Q_1}^{\epsilon_2}(P) h_{P}^{\eta}.
$$
Then, it yields that
\begin{align*}
(S_{\D}\Psi)^2
&=\sum_{P \in \D,\eta \neq 1}|\langle g,h_P^{\eta} \rangle |^2 \Big(\sum_{\substack{Q_1:Q_1 \subsetneq P \\ \epsilon_1,\epsilon_2 \neq 1}}
\langle f_1, h_{Q_1}^{\epsilon_1} \rangle \langle f_2,h_{Q_1}^{\epsilon_2} \rangle |Q_1|^{\frac{\alpha}{n}} h_{Q_1}^{\epsilon_1}(P) h_{Q_1}^{\epsilon_2}(P) \Big)^2 \frac{\mathbbm{1}_P}{|P|} \\
&\lesssim \sum_{P \in \D,\eta \neq 1}|\langle g,h_P^{\eta} \rangle|^2 \Big(\sum_{Q_1 \in \D} |Q_1|^{\frac{\alpha}{n}}
\langle |f_1| \rangle_{Q_1} \langle |f_2| \rangle_{Q_1}\Big)^2 \frac{\mathbbm{1}_P}{|P|} \\
&\leq (S_{\D} g)^2 \cdot \I_{\alpha}(|f_1|,|f_2|)^2.
\end{align*}
Similarly argument as we deal with $\Lambda^1_b$, one may obtain
$$
\big\| \Lambda_b^2(f_1,f_2) \big\|_{L^q(\Rn)}
\lesssim ||b||_{BMO^2_{\D}} \big\| f_1 \big\|_{L^{p_1}(\Rn)} \big\| f_2 \big\|_{L^{p_2}(\Rn)}.
$$

As for the third part, it can be split in the following form
\begin{align*}
\Lambda_b^{3}(f_1,f_2)
&=\sum_{\substack{P,Q_1,Q_2 \in \D \\ Q_2 = P \subsetneq Q_1}} \sum_{\eta,\epsilon_1,\epsilon_2 \neq 1}
+\sum_{\substack{P,Q_1,Q_2 \in \D \\ Q_2 \subsetneq P \subsetneq Q_1}} \sum_{\eta,\epsilon_1,\epsilon_2 \neq 1}
+\sum_{\substack{P,Q_1,Q_2 \in \D \\ P \subsetneq Q_2 \subsetneq Q_1}} \sum_{\eta,\epsilon_1,\epsilon_2 \neq 1} \\
&:= \Lambda_b^{3,1}(f_1,f_2) + \Lambda_b^{3,2}(f_1,f_2) +  \Lambda_b^{3,3}(f_1,f_2).
\end{align*}
Note that
$$
\Lambda_b^{3,3}(f_1,f_2) = \Lambda_b^{1}(f_1,f_2).
$$
From the equality $(\ref{average})$, it follows that
\begin{align*}
\Lambda_b^{3,1}(f_1,f_2)
&=\sum_{P \in \D}\sum_{\eta,\epsilon_2 \neq 1} \langle b,h_P^\eta \rangle \langle f_2,h_P^{\epsilon_2} \rangle |P|^{\frac{\alpha}{n}} h_P^{\epsilon_2} h_P^{\eta}\sum_{\substack{Q_1 \in \D:Q_1 \supsetneq P \\ \epsilon_1 \neq 1}} \langle f_1,h_{Q_1}^{\epsilon_1} \rangle h_{Q_1}^{\epsilon_1}(P) \\
&=\sum_{P \in \D}\sum_{\eta,\epsilon_2 \neq 1} \langle b,h_P^\eta \rangle \langle f_2,h_P^{\epsilon_2} \rangle \langle f_1 \rangle_P |P|^{\frac{\alpha}{n}} h_P^{\epsilon_2} h_P^{\eta} \\
&=\sum_{P \in \D,\eta \neq 1} \langle b,h_P^\eta \rangle \langle f_2,h_P^{\eta} \rangle \langle f_1 \rangle_P |P|^{\frac{\alpha}{n}} \frac{\mathbbm{1}_P}{|P|} \\
&\quad + \sum_{P \in \D}\sum_{\substack{\eta,\epsilon_2 \neq 1 \\ \eta \neq \epsilon_2}} \langle b,h_P^\eta \rangle \langle f_2,h_P^{\epsilon_2} \rangle \langle f_1 \rangle_P |P|^{\frac{\alpha}{n}} |P|^{-\frac12} h_P^{\epsilon_2 + \eta} \\
&:=\Lambda_{b,=}^{3,1}(f_1,f_2) + \Lambda_{b,\neq}^{3,1}(f_1,f_2) .
\end{align*}
It immediately yields that
$$
\big\langle \Lambda_{b,=}^{3,1}(f_1,f_2),g \big\rangle
=\langle b,\Psi_{=} \rangle,\ \
\big\langle \Lambda_{b,\neq}^{3,1}(f_1,f_2),g \big\rangle
=\langle b,\Psi_{\neq} \rangle,
$$
where
\begin{eqnarray*}
\Psi_{=}
&:=&\sum_{P \in \D,\eta \neq 1} \langle f_2,h_{P}^\eta \rangle |P|^{\frac{\alpha}{n}} \langle f_1 \rangle_P \langle g \rangle_P h_P^{\eta}, \\
\Psi_{\neq}
&:=&\sum_{P \in \D}\sum_{\substack{\eta,\epsilon_2 \neq 1 \\ \eta \neq \epsilon_2}} \langle f_2,h_{P}^{\epsilon_2} \rangle
\langle g,h_P^{\eta + \epsilon_2} \rangle |P|^{\frac{\alpha}{n}} \langle f_1 \rangle_P h_P^{\eta}.
\end{eqnarray*}
Thus, we obtain the following pointwise estimates:
\begin{align*}
(S_{\D}\Psi_{=})^2
=\sum_{P \in \D,\eta \neq 1}\big| \langle f_2,h_{P}^\eta \rangle \big|^2 \big(|P|^{\frac{\alpha}{n}} \langle f_1 \rangle_P
\langle g \rangle_P \big)^2 \frac{\mathbbm{1}_P}{|P|}
\leq (S_{\D} f_2)^2 \cdot \M_{\alpha}(f_1,g)^2.
\end{align*}
Moreover,
\begin{align*}
(S_{\D}\Psi_{\neq})^2
&=\sum_{P \in \D,\eta \neq 1} \Big(\sum_{\epsilon_2 \neq 1, \epsilon_2 \neq \eta}  \langle f_2,h_{P}^{\epsilon_2} \rangle
\langle g,h_P^{\eta + \epsilon_2} \rangle |P|^{\frac{\alpha}{n}} \langle f_1 \rangle_P \Big)^2 \frac{\mathbbm{1}_P}{|P|} \\
&\lesssim \sum_{P \in \D} \Big(\sum_{\epsilon_2 \neq 1}  \langle f_2,h_{P}^{\epsilon_2} \rangle |P|^{\frac{\alpha}{n}} \langle |f_1| \rangle_P \langle |g| \rangle_P \Big)^2 \frac{\mathbbm{1}_P}{|P|} \\
&\lesssim \sum_{P \in \D,\epsilon_2 \neq 1} \big|\langle f_2,h_{P}^{\epsilon_2} \rangle\big|^2 \frac{\mathbbm{1}_P}{|P|} \cdot \M_{\alpha}(f_1,g)^2 \\
&\leq (S_{\D} f_2)^2 \cdot \M_{\alpha}(f_1,g)^2.
\end{align*}
Accordingly, by the H\"{o}lder inequality and $(\ref{M-alpha})$, it now follows that 
\begin{equation}\aligned\label{b-3-1}
\Big| \big\langle \Lambda_b^{3,1}(f_1,f_2),g \big\rangle \Big|
&\lesssim  ||b||_{BMO^2_{\D}} \Big(\big\| S_{\D} \Psi_{=} \big\|_{L^1(\Rn)} + \big\| S_{\D} \Psi_{\neq} \big\|_{L^1(\Rn)} \Big) \\
&\lesssim  ||b||_{BMO^2_{\D}} \big\|S_{\D} f_2 \cdot \M_{\alpha}(f_1,g)\big\|_{L^1(\Rn)} \\
&\lesssim  ||b||_{BMO^2_{\D}} \big\|S_{\D} f_2 \big\|_{L^{p_2}(\Rn)} \big\| \M_{\alpha}(f_1,g)\big\|_{L^{p_2'}(\Rn)} \\
&\lesssim  ||b||_{BMO^2_{\D}} \big\| f_2 \big\|_{L^{p_2}(\Rn)} \big\| f_1 \big\|_{L^{p_1}(\Rn)} \big\| g \big\|_{L^{q'}(\Rn)}.
\endaligned
\end{equation}
This gives that
$$
\big\| \Lambda_b^{3,1}(f_1,f_2) \big\|_{L^q(\Rn)}
\lesssim ||b||_{BMO^2_{\D}} \big\| f_1 \big\|_{L^{p_1}(\Rn)} \big\| f_2 \big\|_{L^{p_2}(\Rn)}.
$$

Finally, we consider the estimate of $\Lambda_b^{3,2}$. The average identity $(\ref{average})$ implies that
\begin{align*}
\Lambda_b^{3,2}(f_1,f_2)
&=\sum_{\substack{P,Q_2 \in \D \\ Q_2 \subsetneq P}} \sum_{\eta,\epsilon_2 \neq 1} \langle b,h_P^\eta \rangle \langle f_2, h_{Q_2}^{\epsilon_2} \rangle |Q_2|^{\frac{\alpha}{n}} h_{P}^{\eta}(Q_2) h_{Q_2}^{\epsilon_2}
\sum_{\substack{Q_1 \in \D:Q_1 \supsetneq P \\ \epsilon_1 \neq 1}} \langle f_1, h_{Q_1}^{\epsilon_1} \rangle h_{Q_1}^{\epsilon_1}(P) \\
&=\sum_{\substack{P,Q_2 \in \D \\ Q_2 \subsetneq P}} \sum_{\eta,\epsilon_2 \neq 1} \langle b,h_P^\eta \rangle \langle f_2, h_{Q_2}^{\epsilon_2} \rangle \langle f_1 \rangle_P |Q_2|^{\frac{\alpha}{n}} h_{P}^{\eta}(Q_2) h_{Q_2}^{\epsilon_2}.
\end{align*}
Furthermore, we get
$$
\big\langle \Lambda_{b}^{3,2}(f_1,f_2),g \big\rangle
=\langle b,\psi \rangle
$$
where
$$
\psi=\sum_{\substack{P,Q_2 \in \D \\ Q_2 \subsetneq P}} \sum_{\eta,\epsilon_2 \neq 1} \langle f_2, h_{Q_2}^{\epsilon_2} \rangle \langle g, h_{Q_2}^{\epsilon_2} \rangle \langle f_1 \rangle_P  |Q_2|^{\frac{\alpha}{n}} h_P^{\eta}(Q_2) h_P^{\eta}.
$$
Thus, we obtain
\begin{equation}\aligned\label{SD-psi}
S_{\D}\psi
&=\bigg(\sum_{P \in \D,\eta \neq 1} \Big(\sum_{\substack{Q_2:Q_2 \subsetneq P \\ \epsilon_2 \neq 1}} \langle f_2, h_{Q_2}^{\epsilon_2} \rangle \langle g,h_{Q_2}^{\epsilon_2} \rangle \langle f_1 \rangle_P  |Q_2|^{\frac{\alpha}{n}} h_P^{\eta}(Q_2) \Big)^2 \frac{\mathbbm{1}_P}{|P|} \bigg)^{\frac12} \\
&\leq \sum_{Q_2 \in \D,\epsilon_2 \neq 1} |Q_2|^{\frac{\alpha}{n}} \big|\langle f_2, h_{Q_2}^{\epsilon_2} \rangle \langle g, h_{Q_2}^{\epsilon_2} \rangle
\big| \Big(\sum_{\substack{P \in \D: P \supsetneq Q_2 \\ \eta \neq 1}} \big|\langle f_1 \rangle_P h_P^{\eta}(Q_2)\big|^2 \frac{\mathbbm{1}_P}{|P|} \Big)^{\frac12} \\
&\lesssim M(f_1)\sum_{Q_2 \in \D} |Q_2|^{\frac{\alpha}{n}} \langle |f_2| \rangle_{Q_2} \langle g \rangle_{Q_2} |Q_2|
\Big(\sum_{P:P \supsetneq Q_2} \frac{1}{|P|^2}\Big)^{\frac12} \\
&\lesssim M(f_1) \cdot \I_{\alpha}(|f_2|,|g|).
\endaligned
\end{equation}
Proceeding as we did in $(\ref{b-3-1})$, it follows that
$$
\big\| \Lambda_b^{3,2}(f_1,f_2) \big\|_{L^q(\Rn)}
\lesssim ||b||_{BMO^2_{\D}} \big\| f_1 \big\|_{L^{p_1}(\Rn)} \big\| f_2 \big\|_{L^{p_2}(\Rn)}.
$$
\qed
\subsection{Estimate of $\Delta_b$}
In this subsection, we will deal with $\Delta_b$. Thanks to $(\ref{average})$, we have
\begin{align*}
\Delta_b(f_1,f_2)
&=\sum_{\substack{P,Q_2 \in \D \\ P \cap Q_2 \neq \emptyset}} \sum_{\eta,\epsilon_2 \neq 1} \langle b, h_P^{\eta} \rangle \langle f_2, h_{Q_2}^{\epsilon_2} \rangle
|P \cap Q_2|^{\frac{\alpha}{n}} h_P^{\eta} h_{Q_2}^{\epsilon_2} \sum_{\substack{Q_1 \in \D: Q_1 \supsetneq P \\ \epsilon_1 \neq 1}} \langle f_1, h_{Q_1}^{\epsilon_1} \rangle  h_{Q_1}^{\epsilon_1}(P) \\
&=\sum_{\substack{P,Q_2 \in \D \\ P \cap Q_2 \neq \emptyset}} \sum_{\eta,\epsilon_2 \neq 1} \langle b, h_P^{\eta} \rangle \langle f_2, h_{Q_2}^{\epsilon_2} \rangle \langle f_1 \rangle_{P} |P \cap Q_2|^{\frac{\alpha}{n}} h_P^{\eta} h_{Q_2}^{\epsilon_2} \\
&=\sum_{\substack{P,Q_2 \in \D \\ Q_2 = P}}\sum_{\eta,\epsilon_2 \neq 1}
+ \sum_{\substack{P,Q_2 \in \D \\ P \subsetneq Q_2}} \sum_{\eta,\epsilon_2 \neq 1}
+ \sum_{\substack{P,Q_2 \in \D \\ Q_2 \subsetneq P}} \sum_{\eta,\epsilon_2 \neq 1} \\
&:=\Delta_{b,1}(f_1,f_2) + \Delta_{b,2}(f_1,f_2) + \Delta_{b,3}(f_1,f_2).
\end{align*}
It is easy to see that
$$
\Delta_{b,3}(f_1,f_2) = \Lambda_b^{3,2}(f_1,f_2),
$$
and
$$
\Delta_{b,1}(f_1,f_2)
=\sum_{P \in \D} \sum_{\substack{\eta,\epsilon_2 \neq 1 \\ \eta=\epsilon_2}}
+\sum_{P \in \D} \sum_{\substack{\eta,\epsilon_2 \neq 1 \\ \eta \neq \epsilon_2}}
=\Lambda_{b,=}^{3,1}(f_1,f_2) + \Lambda_{b,\neq}^{3,1}(f_1,f_2).
$$
Thus, it suffices to bound $\Delta_{b,2}$. Indeed, there holds that
\begin{align*}
\Delta_{b,2}(f_1,f_2)
&=\sum_{P \in \D,\eta \neq 1} \langle b, h_P^{\eta} \rangle |P|^{\frac{\alpha}{n}} \langle f_1 \rangle_{P} h_P^{\eta}
\sum_{\substack{Q_2 \in : Q_2 \supsetneq P \\ \epsilon_2 \neq 1}} \langle f_2, h_{Q_2}^{\epsilon_2} \rangle h_{Q_2}^{\epsilon_2}(P) \\
&=\sum_{P \in \D,\eta \neq 1} \langle b, h_P^{\eta} \rangle |P|^{\frac{\alpha}{n}} \langle f_1 \rangle_{P} \langle f_2 \rangle_{P} h_P^{\eta}.
\end{align*}
Then it yields that
\begin{align*}
\big\langle \Delta_{b,2}(f_1,f_2),g \big\rangle
=\Big\langle b,\sum_{P \in \D,\eta \neq 1} \langle g, h_P^{\eta} \rangle |P|^{\frac{\alpha}{n}} \langle f_1 \rangle_{P} \langle f_2 \rangle_{P} h_P^{\eta} \Big\rangle
:=\langle b,\phi \rangle.
\end{align*}
Notice that
\begin{align*}
(S_D \phi)^2
&=\sum_{P \in \D,\eta \neq 1} \big|\langle g, h_P^{\eta} \rangle\big|^2 \Big(|P|^{\frac{\alpha}{n}} \langle f_1 \rangle_{P}
\langle f_2 \rangle_{P}\Big)^2 \frac{\mathbbm{1}_P}{|P|} \\
&\leq (S_{\D} g)^2 \cdot \M_{\alpha}(f_1,f_2)^2,
\end{align*}
which parallels with $(\ref{I-alpha-f1-f2})$. By Theorem $\ref{M-alpha-I-alpha}$, it yields that
$$
\big\| \Delta_{b,2}(f_1,f_2) \big\|_{L^q(\Rn)}
\lesssim ||b||_{BMO^2_{\D}} \big\| f_1 \big\|_{L^{p_1}(\Rn)} \big\| f_2 \big\|_{L^{p_2}(\Rn)}.
$$
\qed
\subsection{Estimate of $\Xi_b$}
In order to get the two-weight inequality of $\Xi_b$, we only need to note that
\begin{align*}
\Xi_b(f_1,f_2)
&=\sum_{Q_1 \in \D,\epsilon_1 \neq 1} \langle b,h_{Q_1}^{\epsilon_1} \rangle \langle f_1,h_{Q_1}^{\epsilon_1} \rangle |Q_1|^{\frac{\alpha}{n}} \Big(\sum_{\substack{Q_2 \in \D:Q_2 \supsetneq Q_1 \\ \epsilon_2 \neq 1}} \langle f_2,h_{Q_2}^{\epsilon_2} \rangle h_{Q_2}^{\epsilon_2}(Q_1)\Big) \frac{\mathbbm{1}_{Q_1}}{|Q_1|} \\
&=\sum_{Q_1 \in \D,\epsilon_1 \neq 1} \langle b,h_{Q_1}^{\epsilon_1} \rangle \langle f_1,h_{Q_1}^{\epsilon_1} \rangle \langle f_2 \rangle_{Q_1} |Q_1|^{\frac{\alpha}{n}} \frac{\mathbbm{1}_{Q_1}}{|Q_1|} \\
&=\Lambda_{b,=}^{3,1}(f_2,f_1).
\end{align*}
\qed
\subsection{Estimate of $\Theta_b$}
Using averaging identity $(\ref{average})$ again, we obtain
\begin{align*}
\Theta_b(f_1,f_2)
&=\sum_{Q_1 \in \D,\epsilon_1 \neq 1} \langle b,h_{Q_1}^{\epsilon_1} \rangle \langle f_1,h_{Q_1}^{\epsilon_1} \rangle \sum_{Q \in \D:Q \supsetneq Q_1}
|Q|^{\frac{\alpha}{n}} \Big(\sum_{\substack{Q_2 \in \D: Q_2 \supsetneq Q \\ \epsilon_2 \neq 1}} \langle f_1,h_{Q_2}^{\epsilon_2} \rangle h_{Q_2}^{\epsilon_2}(Q)\Big) \frac{\mathbbm{1}_{Q}}{|Q|} \\
&=\sum_{Q_1 \in \D,\epsilon_1 \neq 1} \langle b,h_{Q_1}^{\epsilon_1} \rangle \langle f_1,h_{Q_1}^{\epsilon_1} \rangle \sum_{Q \in \D:Q \supsetneq Q_1}
|Q|^{\frac{\alpha}{n}} \langle f_2 \rangle_Q \frac{\mathbbm{1}_{Q}}{|Q|}.
\end{align*}
Hence, we have
\begin{align*}
\big\langle \Theta_b(f_1,f_2),g \big\rangle
=\langle b,\gamma \rangle
:= \Big\langle b, \sum_{Q_1 \in \D,\epsilon_1 \neq 1} \langle f_1,h_{Q_1}^{\epsilon_1} \rangle \sum_{Q \in \D:Q \supsetneq Q_1}
|Q|^{\frac{\alpha}{n}} \langle f_2 \rangle_Q \langle g \rangle_Q h_{Q_1}^{\epsilon_1}\Big\rangle,
\end{align*}
and
\begin{align*}
(S_D \gamma)^2
&= \sum_{Q_1 \in \D,\epsilon_1 \neq 1} \big|\langle f_1,h_{Q_1}^{\epsilon_1} \rangle\big|^2 \Big(\sum_{Q \in \D:Q \supsetneq Q_1}
|Q|^{\frac{\alpha}{n}} \langle f_2 \rangle_Q \langle g \rangle_Q \Big)^2 \frac{\mathbbm{1}_{Q_1}}{|Q_1|} \\
&\lesssim \sum_{Q_1 \in \D,\epsilon_1 \neq 1} \big|\langle f_1,h_{Q_1}^{\epsilon_1} \rangle\big|^2 \Big(\sum_{Q \in \D}
|Q|^{\frac{\alpha}{n}} \langle |f_2| \rangle_Q \langle |g| \rangle_Q \Big)^2 \frac{\mathbbm{1}_{Q_1}}{|Q_1|} \\
&\leq (S_{\D} f_1)^2 \cdot \I_{\alpha}(|f_2|,|g|)^2.
\end{align*}
This inequality is similar to $(\ref{SD-psi})$, since Lemma $\ref{SD-f}$ holds. Therefore, one may obtain that
$$
\big\| \Theta_{b}(f_1,f_2) \big\|_{L^q(\Rn)}
\lesssim ||b||_{BMO^2_{\D}} \big\| f_1 \big\|_{L^{p_1}(\Rn)} \big\| f_2 \big\|_{L^{p_2}(\Rn)}.
$$
\qed
\section{Characterization of Bilinear Commutator}
By symmetry, it is enough to show the result for $[b,\I_{\alpha}^{\D}]_1$.
\subsection{Upper Bound for Bilinear Commutator}
Applying the decomposition with respect to Haar functions $(\ref{decomposition})$, we get
$$
[b,\I_\alpha^{\D}]_1(f_1,f_2)
=\sum_{P,Q_1,Q_2 \in \D} \sum_{\eta,\epsilon_1,\epsilon_2 \neq 1} \langle b,h_P^{\eta} \rangle \langle f_1, h_{Q_1}^{\epsilon_1} \rangle
\langle f_2,h_{Q_2}^{\epsilon_2} \rangle [h_P^{\eta},\I_{\alpha}^{\D}]_1(h_{Q_1}^{\epsilon_1},h_{Q_2}^{\epsilon_2}).
$$
First, one may Observe that
\begin{align*}
\I_{\alpha}^{\D}(h_{Q_1}^{\epsilon_1},h_{Q_2}^{\epsilon_2})
&=\sum_{Q \in \D} |Q|^{\frac{\alpha}{n}} \langle h_{Q_1}^{\epsilon_1} \rangle_Q
\langle h_{Q_2}^{\epsilon_2} \rangle_Q \mathbbm{1}_Q \\
&=\Big( \sum_{\substack{Q \in \D \\ Q \subsetneq Q_1 \cap Q_2}} |Q|^{\frac{\alpha}{n}} \mathbbm{1}_Q \Big)
h_{Q_1}^{\epsilon_1}  h_{Q_2}^{\epsilon_2} \\
&=c_\alpha |Q_1 \cap Q_2|^{\frac{\alpha}{n}} h_{Q_1}^{\epsilon_1}  h_{Q_2}^{\epsilon_2}.
\end{align*}
A simple calculation yields that
\begin{align*}
\I_{\alpha}^{\D}(\mathbbm{1}_{Q_1},h_{Q_2}^{\epsilon_2})
&=\sum_{Q \in \D:Q \subsetneq Q_2} |Q|^{\frac{\alpha}{n}} \frac{|Q_1 \cap Q|}{|Q|}
h_{Q_2}^{\epsilon_2} \mathbbm{1}_Q \\
&=|Q_1|^{\frac{\alpha}{n}} h_{Q_2}^{\epsilon_2} \mathbbm{1}_{Q_1} \mathbbm{1}_{\{Q_1 \subsetneq Q_2\}}
+ \Big( \sum_{\substack{Q \in \D \\ Q \subsetneq Q_1 \cap Q_2}} |Q|^{\frac{\alpha}{n}} \mathbbm{1}_Q \Big) h_{Q_2}^{\epsilon_2} \\
&\quad\quad + |Q_1| \sum_{Q:Q_1 \subsetneq Q \subsetneq Q_2} |Q|^{\frac{\alpha}{n}} h_{Q_2}^{\epsilon_2}  \frac{\mathbbm{1}_Q}{|Q|} \\
&=(1+c_\alpha)|Q_1 \cap Q_2|^{\frac{\alpha}{n}} h_{Q_2}^{\epsilon_2} \mathbbm{1}_{Q_1}
+ |Q_1| \sum_{Q:Q_1 \subsetneq Q \subsetneq Q_2} |Q|^{\frac{\alpha}{n}} h_{Q_2}^{\epsilon_2}  \frac{\mathbbm{1}_Q}{|Q|}.
\end{align*}
Thus, we have
\begin{align*}
&[h_P^{\eta},\I_{\alpha}^{\D}]_1(h_{Q_1}^{\epsilon_1},h_{Q_2}^{\epsilon_2}) \\
&=\begin{cases}
c_\alpha h_P^{\eta} h_{Q_1}^{\epsilon_1}  h_{Q_2}^{\epsilon_2}\big(|Q_1 \cap Q_2|^{\frac{\alpha}{n}}
- |P \cap Q_2|^{\frac{\alpha}{n}}\big),\ \ &\text{if}\ P \subsetneq Q_1;\\
-|Q_1|^{\frac{\alpha}{n}} h_{Q_2}^{\epsilon_2} \mathbbm{1}_{\{Q_1 \subsetneq Q_2\}} \frac{\mathbbm{1}_{Q_1}}{|Q_1|}
- \sum\limits_{Q:Q_1 \subsetneq Q \subsetneq Q_2} |Q|^{\frac{\alpha}{n}} h_{Q_2}^{\epsilon_2}  \frac{\mathbbm{1}_Q}{|Q|},\ \
&\text{if}\ P = Q_1 \text{ and}\ \epsilon_1 = \eta \\
0,\ \ &\text{otherwise}.
\end{cases}
\end{align*}
Consequently, it yields that
$$
[b,\I_\alpha^{\D}]_1(f_1,f_2)
=c_\alpha \Lambda_b(f_1,f_2)
-c_\alpha \Delta_b(f_1,f_2)
-\Xi_b(f_1,f_2)
-\Theta_b(f_1,f_2).
$$
From Proposition $\ref{Pro-Pra}$ and the dyadic version of the inequality $(\ref{BMO-BMO})$, it follows that
$$
\big\| [b,\I_\alpha^{\D}]_1 \big\|_{L^{p_1}(\Rn) \times L^{p_2}(\Rn) \rightarrow L^q(\Rn)}
\lesssim ||b||_{BMO_{\D}}.
$$
\qed
\subsection{Lower Bound for Bilinear Commutator}
We will follow the scheme of the proof in \cite{L} to track the precise constants. Let $||b||_{BMO_{\D}^2}=1$ and the dyadic cube $J$ satisfying
$$
\frac{1}{|J|}\sum_{\substack{I: I \subset J \\ \epsilon \neq 1}} |\langle b,h_I^{\epsilon}\rangle|^2 \geq \frac12.
$$
By the John-Nirenberg estimates, for $\frac{1}{p}=\frac{1}{p_1} + \frac{1}{p_2}$, we have
$$
1 \lesssim |J|^{-1/p} \Big\| \sum_{\substack{I: I \subset J \\ \epsilon \neq 1}} \langle b,h_I^{\epsilon} \rangle  h_I^{\epsilon} \Big\|_{L^p(\Rn)}
  \lesssim |J|^{-1/q} \Big\| \sum_{\substack{I: I \subset J \\ \epsilon \neq 1}} \langle b,h_I^{\epsilon} \rangle  h_I^{\epsilon} \Big\|_{L^q(\Rn)}
  \lesssim 1.
$$
We split the function $b=b'+b''$, where
$$
b'=\sum_{\substack{I: \ell(I) \leq \ell(J) \\ \epsilon \neq 1}} \langle b,h_I^{\epsilon}\rangle h_I^{\epsilon}.
$$
Then, for any $x \in J$, we get
$$
[b,\I_\alpha^{\D}]_1(\mathbbm{1}_J, \mathbbm{1}_J)(x)
=b' \I_\alpha^{\D}(\mathbbm{1}_J, \mathbbm{1}_J)(x) - \I_\alpha^{\D}(b' \mathbbm{1}_J, \mathbbm{1}_J)(x).
$$
Furthermore, it is easy to check that
$$
b' \I_\alpha^{\D}(\mathbbm{1}_J, \mathbbm{1}_J)(x)
= (1+c_{\alpha}) |J|^{\frac{\alpha}{n}} \sum_{\substack{I: I \subset J \\ \epsilon \neq 1}} \langle b,h_I^{\epsilon}\rangle h_I^{\epsilon},
$$
and
$$
\I_\alpha^{\D}(b' \mathbbm{1}_J, \mathbbm{1}_J)(x)
= (1+c_{\alpha}) \sum_{\substack{I: I \subset J \\ \epsilon \neq 1}} |I|^{\frac{\alpha}{n}} \langle b,h_I^{\epsilon}\rangle h_I^{\epsilon}.
$$
Thus, we have
$$
\big\| [b,\I_\alpha^{\D}]_1(\mathbbm{1}_J, \mathbbm{1}_J) \big\|_{L^q(\Rn)}
\gtrsim |J|^{\frac{\alpha}{n}} \Big\| \sum_{\substack{I: I \subset J \\ \epsilon \neq 1}} \langle b,h_I^{\epsilon}\rangle h_I^{\epsilon} \Big\|_{L^q(\Rn)}
\gtrsim |J|^{\frac{\alpha}{n}+\frac1q}.
$$
On the other hand, it holds that
$$
\big\| [b,\I_\alpha^{\D}]_1(\mathbbm{1}_J, \mathbbm{1}_J) \big\|_{L^q(\Rn)}
\lesssim |J|^{\frac1p} \big\| [b,\I_\alpha^{\D}]_1 \big\|_{L^{p_1}(\Rn) \times L^{p_2}(\Rn) \rightarrow L^q(\Rn)} .
$$
Therefore, we deduce that
$$
\big\| [b,\I_\alpha^{\D}]_1 \big\|_{L^{p_1}(\Rn) \times L^{p_2}(\Rn) \rightarrow L^q(\Rn)}
\gtrsim 1 = ||b||_{BMO_{\D}^2} \simeq ||b||_{BMO_{\D}}.
$$
\qed

\end{document}